%% file: main.tex
\RequirePackage{plautopatch}
\documentclass[dvipdfmx,12pt]{amsart}

\usepackage[top=25truemm,bottom=20truemm,left=20truemm,right=20truemm]{geometry}
\usepackage[noadjust]{cite}
\usepackage[dvipsnames]{xcolor}
\usepackage{enumitem}
\usepackage[dvipdfmx]{graphicx} %画像の操作
\usepackage{amsmath,amssymb,amsfonts,amssymb,amscd} %数式関係
\usepackage{amsthm} %定理環境
\usepackage{thmtools}
\usepackage[all]{xy} %射の図式
\usepackage{mathtools}
\usepackage[dvipdfmx]{hyperref}
\usepackage{cleveref} % amsmathの後必須

\theoremstyle{plain}
\newtheorem{thm}{Theorem}[section]
\newtheorem{cor}[thm]{Corollary}

\newtheorem{prop}[thm]{Proposition}
\newtheorem{ques}[thm]{Question}
\theoremstyle{definition}
\newtheorem{dfn}[thm]{Definition}
\newtheorem{ex}[thm]{Example}
\newtheorem{rmk}[thm]{Remark}
\crefname{thm}{Theorem}{Theorems}
\crefname{cor}{Corollary}{Corollarys}
\crefname{lem}{Lemma}{Lemmas}
\crefname{prop}{Proposition}{Propositions}
\crefname{dfn}{Definition}{Definitions}
\crefname{ex}{Example}{Examples}
\crefname{rmk}{Remark}{Remarks}
\crefname{ques}{Question}{Questions}
\crefname{equation}{equation}{equations}

\newcommand{\As}{\mathrm{As}}
\newcommand{\ab}{\mathrm{ab}}
\newcommand{\qop}{\triangleleft}
\newcommand{\Conj}{\mathrm{Conj}}
\newcommand{\Aut}{\mathrm{Aut}}
\newcommand{\Inn}{\mathrm{Inn}}
\newcommand{\Qdl}{\mathrm{Qdl}}
\newcommand{\Grp}{\mathrm{Grp}}
\newcommand{\GAlex}{\mathrm{GAlex}}

\newcommand{\OO}{\mathcal{O}}
\date{\today}
\title{A condition of admissibility for generalized Alexander quandles} % タイトル
\author{Katsunori Arai, Keisuke Himeno, Ryoya Kai, and Yuko Ozawa} % 著者名
\address{
	(K. Arai)
    Kyoto Meitoku High School,
    3-8 Oehigashinagacho,
    Nishikyo-ku, Kyoto 610-1111, Japan
}
\email{arai1223math@gmail.com}

\address{
	(K. Himeno) 
    Faculty of Education,
    Yamaguchi University, 
    Yoshida, 
	Yamaguchi, 753-8511, Japan}
\email{himekei@yamaguchi-u.ac.jp}

\address{
	(R. Kai) 
    Center for Educational Research of Science and Mathematics,
	Nara University of Education, 
	Takabatake-Cho, Nara City, 
	Nara, 630-8301, Japan}
\email{kai.ryoya.d8@cc.nara-edu.ac.jp}

\address{
    (Y. Ozawa) 
    Independent Scholar}
\email{y.math46271246@gmail.com}

\begin{document}
%%% --- %%%
\input{0_abstract.tex}

\maketitle
%%%%%%%%%%%
%%%%%%%%%%%
\input{1_introduction.tex}
\input{2_notions.tex}
\input{3_galex.tex}
%%%%%%%%%%%
\section*{Acknowledgment}
The third author would like to express their sincere gratitude to Professors Ali Baklouti and Hideyuki Ishi, the organizers of the 8th Tunisian-Japanese Conference “Geometric and Harmonic Analysis on Homogeneous Spaces and Applications”.
The authors are also grateful
to Seiichi Kamada, Hiroaki Nagaya, Takayuki Okuda, Hiroshi Tamaru and Yuta Taniguchi for helpful comments and useful discussions.
The second author was supported by JST SPRING, Grant Number JPMJSP2132.
The third author was supported by JST SPRING, Grant Number JPMJSP2139,
and was partly supported by MEXT Promotion of Distinctive Joint Research Center Program JPMXP0723833165 and Osaka Metropolitan University Strategic Research Promotion Project (Development of International Research Hubs).
%%----%%
\bibliographystyle{alpha}
\bibliography{AdmissibleGAlex.bib}
%-----------------------------

%Please write the address of the author. 
	
\end{document}

%% file: 0_abstract.tex
\begin{abstract}
A typical example of a quandle is the conjugation quandle.
A quandle is said to be admissible if it is isomorphic to a conjugation quandle.
We study the admissibility problem for quandles, that is, determining whether a given quandle is admissible.
% isomorphic to a conjugation quandle.
In particular,
we focus on generalized Alexander quandles, 
which are groups equipped with quandle structures defined by group automorphisms.
% Focusing on generalized Alexander quandles associated with a group and its automorphism, 
In this paper,
we provide a sufficient condition for admissibility in terms of geometric properties of quandles:
% algebraic connectedness and antipodal sets: 
namely, if every antipodal set in each algebraically connected component consists of a single point, then the quandle is admissible. 
This result generalizes previous results on generalized Alexander quandles.
\end{abstract}

%% file: 1_introduction.tex
\section{Introduction}\label{sec:Intro}

% * カンドルは群の共役演算を公理化した二項演算を持つ代数系．
A quandle is an algebraic system originated from knot theory.
The algebraic structure is regarded as a generalization of the conjugation of groups.
The definition is as follows:
% - def: quandle
\begin{dfn}[\cite{Joyce-1982-ClassifyingInvariantKnotsKnota,Matveev-1982-DistributiveGroupoidsKnotTheorya}]\label{def:quandle}
  A \emph{quandle} is a non-empty set $X$ equipped with a binary operation $\qop$ satisfying the following properties:
  \begin{enumerate}
    \item[Q1] $x \qop x = x$ for any $x \in X$.

    \item[Q2] For any $y \in X$, the right multiplication $s_y: X \to X$ defined by $s_y(x) := x \qop y$ is a bijection. 

    \item[Q3] $(x \qop y) \qop z = (x \qop z) \qop (y \qop z)$ for any $x, y, z \in X$.
  \end{enumerate}
\end{dfn}

Any symmetric space becomes a quandle by its point symmetries \cite{Loos-1969-SymmetricSpacesGeneralTheorya,Joyce-1982-ClassifyingInvariantKnotsKnota}.
Thus, the map $s_y: X \to X$ is called a \emph{point symmetry} at $y \in X$.
Since quandle structures can be regarded as a generalization of the algebraic structure of point symmetries of symmetric spaces,
they are often regarded as discretizations of symmetric spaces.
The notions of homomorphisms, isomorphisms, and automorphisms are naturally defined.
We note that any point symmetry is a quandle automorphism by the axioms (Q2) and (Q3).
The group generated by the set of all point symmetries is called the \emph{inner automorphism group}, and is denoted by $\Inn(X)$. 
A subset $A$ of a quandle $X$ is a \emph{subquandle} if it satisfies $s_y^{\pm 1}(A) = A$ for any $y \in A$.
A typical example of quandle structures is given by the conjugation of groups.
% - ex: conjugation quandle
\begin{ex}\label{ex:conj}
  A group $G$ equipped with the binary operation $\qop$ defined by $x \qop y:=y^{-1} x y$ becomes a quandle denoted by $\Conj(G)$.
  A subquandle of $\Conj(G)$ for a group $G$ is called a \emph{conjugation quandle}.
  %A quandle is said to be \emph{admissible} if the quandle is isomorphic to some conjugation quandle.
  Note that a conjugation quandle consists of a union of conjugacy classes.
\end{ex}
This example motivates the following definition.
\begin{dfn}\label{def:admissible}
A quandle is said to be \emph{admissible} if the quandle is isomorphic to some conjugation quandle.
\end{dfn}
The term ``admissible'' is due to \cite{Kamada-2005-EnvelopingMonoidalQuandlesa}. See also Proposition~\ref{prop:adm_vs_conj}.

In knot theory, the fundamental group $G(K)$ is a classical and strong invariant and is often investigated by the representation to specific groups.
Similarly, we can define the fundamental quandle $Q(K)$ for each knot $K$, and homomorphisms from $Q(K)$ to a specific quandle $X$ provide invariants for $K$.
In particular, homomorphisms from $Q(K)$ provide stronger invariants than those  obtained from group representations of $G(K)$.
To obtain such a strong invariant, we must have a quandle which is not admissible.
Thus, we suggest the following question.

\begin{ques}[cf. {\cite[Question 3.1]{Bardakov-2017-AutomorphismGroupsQuandlesArising}}]\label{quaestion}
  Determine whether given quandles are admissible.
\end{ques}

We now give a simple sufficient condition for admissibility.
A quandle $X$ is said to be \emph{faithful} if the map $s: X \to \Inn(X)$ defined by $s(y):=s_y$, where $\Inn(X)$ denotes the group generated by all point symmetries, is injective.
Since the map $s: X \to \Conj(\Inn(X))$ is a quandle homomorphism by the third axiom of quandles, 
a faithful quandle $X$ is isomorphic to the subquandle of $\Conj(\Inn(X))$.
This shows that every faithful quandle is admissible (see also \cite[Lemma 2.10]{Grana-2011-NicholsAlgebras}.

In this paper, we focus on the admissibility of a special class of quandles, namely \emph{generalized Alexander quandles}, which play an important role in quandle theory.

\begin{dfn}\label{def:GAlex}
  Let $G$ be a group, and let $\sigma: G \to G$ be a group automorphism.
  Then, the binary operation $\qop$ on $G$ defined by $x \qop y := \sigma(xy^{-1})y$ is a quandle structure on $G$.
  The quandle $(G, \qop)$ is called the \emph{generalized Alexander quandle} and is denoted by $\GAlex(G, \sigma)$.
\end{dfn}

Some results on the admissibility of such quandles are already known:
Akita \cite{Akita-2023-EmbeddingAlexanderQuandlesGroupsa} proved that if $G$ is abelian, then the quandle $\GAlex(G, \sigma)$ is admissible for any group automorphism $\sigma: G \to G$.
Dhanwani--Raundal--Singh 
\cite[Proposition 3.12 (arXiv version)]{Dhanwani-2023-DehnQuandlesGroupsOrientable}
also gave the condition for generalized Alexander quandles to be faithful.

% Generalized Alexander quandles play an important role in quandle theory.
% For example, any homogeneous quandle, which is a quandle whose automorphism group transitively acts, is a quotient of some generalized Alexander quandle.
% It has already known some results for the admissibility of such quandles:
%Dhanwani--Raundal--Singh \cite[\Kai{Proposition 3.12 in arXiv ver.}]{Dhanwani-2023-DehnQuandlesGroupsOrientable}
 %give a condition of faithfulness for generalized Alexander quandles.
% \Himeno{Dhanwani--Raundal--Singh \cite[\Kai{Proposition 3.12 in arXiv ver.}]{Dhanwani-2023-DehnQuandlesGroupsOrientable}
% show that if $\sigma$ is a fixed-point free automorphism, then $\GAlex(G,\sigma)$ is admissible.
% (In fact, it is easy to see that $\sigma$ is fixed-point free if and only if $\GAlex(G,\sigma)$ is faithful. Hence, this follows from the above discussion.)}
%Gra\~na--Heckenberger--Vendramin \cite[Lemma 2.10]{Grana-2011-NicholsAlgebras} show that every faithful quandle is admissible. Moreover, \cite[\Kai{Proposition 3.12 in arXiv ver.}]{Dhanwani-2023-DehnQuandlesGroupsOrientable} reformulates this condition in the context of generalized Alexander quandles $\GAlex(G, \sigma)$ in terms of $\sigma$.} %automorphisms of the underlying group.
% Akita \cite{Akita-2023-EmbeddingAlexanderQuandlesGroupsa} shows that if $G$ is abelian, then the quandle $\GAlex(G, \sigma)$ is admissible for any group automorphism $\sigma: G \to G$.

In this paper, we give a new sufficient condition of admissibility for generalized Alexander quandles. 
The statement of the main result consists of some geometric notions for quandles, the connectedness (\cref{def:AlgConn}) and the antipodal set (\cref{def:antipodal}).

\begin{thm}\label{thm:main_thm}
  Let $X := \GAlex(G, \sigma)$ be a generalized Alexander quandle for a group $G$ and its group automorphism $\sigma: G \to G$.
  % If for any algebraic connected component, every antipodal set consists of a single point, then the quandle is admissible.
  % \Kai{If any algebraic connected component and any maximal antipodal set intersects a single point, then the quandle is admissible.}
  Let $P$ be the algebraic connected component containing the identity element $e \in G$,
  and let $H$ be the maximal antipodal set containing the identity element $e \in G$.
  If $H \cap P = \{e\}$, then the quandle $X$ is admissible.
\end{thm}

This result can be regarded as a generalization of that in \cite{Dhanwani-2023-DehnQuandlesGroupsOrientable}.
% In fact, a generalized Alexander quandle is faithful if and only if any maximal antipodal set consists of a single point.
A proof of the main theorem is given in \cref{sec:GAlex}.

%% file: 2_notions.tex
\section{Some notions for quandles}

In this section, we introduce some notions for quandles.
First, we recall the definition of the associated group.
This group will be used to characterize the admissibility of quandles.
See \cite[\S 8.8]{Kamada-2017-SurfaceKnots4Spacea} for details.
% - def: associated group
\begin{dfn}\label{def:associated_group}
  For a quandle $X$, the \emph{associated group} $\As(X)$ of $X$ is the group defined by the following group presentation:
  \[
    \As(X) := \langle w_x \, (x \in X) \mid w_{x \qop y} = w_y^{-1} w_x w_y \, (x, y \in X)\rangle.
  \]
\end{dfn}
The map $\eta_X: X \to \As(X)$ is defined by $\eta_X(x) := w_x$ for $x \in X$, 
and  is called the \emph{natural map}.
This map plays an important role in this paper.
It is known that assigning a quandle $X$ to its associated group $\As(X)$ defines a functor from the category $\Qdl$ of quandles to the category $\Grp$ of groups \cite[Proposition 8.8.3]{Kamada-2017-SurfaceKnots4Spacea}.
The next proposition shows that the functor $\As: \Qdl \to \Grp$ is the left adjoint to the functor $\Conj: \Grp \to \Qdl$.
% - prop:universality_As
\begin{prop}[{\cite[Proposition 8.8.4]{Kamada-2017-SurfaceKnots4Spacea}}]\label{prop:adjointness_As}
  Let $X$ be a quandle, and let $G$ be a group.
  Then, any quandle homomorphism $f: X \to \Conj(G)$ uniquely induces a group homomorphism $f_*: \As(X) \to G$ such that the following diagram commutes:
  \[
  \xymatrix{
    X \ar[d]^-f \ar[r]^-{\eta_X} & \As(X) \ar[d]^-{f_*} \\
    \Conj(G) \ar@{=}[r]& G
  }
  \]
\end{prop}

% - def: admissibility
%\begin{dfn}[\cite{Kamada-2005-EnvelopingMonoidalQuandlesa}]\label{def:admissible}
%  A quandle $X$ is said to be \emph{admissible} if the natural map $\eta_X: X \to \As(X)$ is injective.
%\end{dfn}

Applying \cref{prop:adjointness_As}, we have the following proposition.
% - prop: admissible vs conjugation
\begin{prop}\label{prop:adm_vs_conj}
  For a quandle $X$, the following are equivalent:
  \begin{enumerate}
    \item[$1$] The quandle $X$ is admissible.
%    \item There exists an injective quandle homomorphism $X \to \Conj(G)$ for some group $G$.
     \item[$2$] The natural map $\eta_X: X \to \As(X)$ is injective.
  \end{enumerate}
\end{prop}

In \cite{Kamada-2005-EnvelopingMonoidalQuandlesa}, the admissibility of quandles is defined as the condition of \cref{prop:adm_vs_conj} (2).

% - faithfulの定義
%We can also define a map $s: X \to \Inn(X)$ by $s(y):=s_y$, where $\Inn(X)$ denotes the group generated by all point symmetries.
%\Kai{By the third axiom of quandles, the map $s: X \to \Conj(\Inn(X))$ is a quandle homomorphism.}
%A quandle $X$ is said to be \emph{faithful} if the map $s: X \to \Inn(X)$ is injective.
%\Himeno{As recalled in Section~\ref{sec:Intro}, it follows from \cite[Lemma 2.10]{Grana-2011-NicholsAlgebras} that every faithful quandle is admissible.}
%\KaiNote{命題から直ちに従うのにGHVを参照すると混乱しそうなので少し書き換えました．}
%\Kai{By \cref{prop:adm_vs_conj}, evety faithful quandle is admissible (see also \cite[Lemma 2.10]{Grana-2011-NicholsAlgebras})}.
%Since the map $s: X \to \Inn(X)$ is a quandle homomorphism,
%any faithful quandle is admissible. 

%Next, 
We introduce some geometric notions for quandles.
We note that the associated group $\As(X)$ acts on the quandle $X$ by $x \cdot w_y := x \qop y$ for $x \in X$ and $w_y \in \As(X)$.
We now recall the definition of algebraic connectivity for quandles.

% - def: connectivity of quandles(by As)
\begin{dfn}\label{def:AlgConn}
  Let $X$ be a quandle.
  An orbit of the action of $\As(X)$ on $X$ is called an \emph{algebraic connected component} of $X$.
  We denote the set of all connected components by $\pi_0(X)$.
  The quandle $X$ is said to be \emph{algebraically connected} if the action is transitive.
\end{dfn}
Note that connected components of a quandle are subquandles. 
Since the action of $\As(X)$ factors through the action of inner automorphism group, the notion of algebraic connectivity is sometimes defined by using the action of the inner automorphism group.
This notion has a relationship with admissibility as follows.

% - lem: admissibility vs connected component
\begin{prop}\label{prop:AlgConn_vs_Adm}
  Let $x$ and $y$ be elements in  a quandle $X$.
  If $\eta_X(x) = \eta_X(y)$ in $\As(X)$, 
  then the elements $x$ and $y$ are contained in the same algebraically connected component.
\end{prop}
\begin{proof}
  We recall that the abelianization $\As(X)_{\ab}$ of $\As(X)$ is isomorphic to the free abelian group generated by $\pi_0(X)$.
  In particular, the element $[\eta_X(x)]$ in $\As(X)_{\ab}$ for $x \in X$ is equal to the element $1 \mathcal{O}_x$, where $\mathcal{O}_x$ is the algebraic connected component containing $x$.
  
  We assume that $\eta_X(x) = \eta_X(y)$ in $\As(X)$ for $x, y \in X$.
  Thus, we have $[\eta_X(x)] = [\eta_X(y)] \As(X)_{\ab}$,
  and hence $\OO_x = \OO_y$, as desired.
\end{proof}

Next, we introduce another geometric notion of quandles originated from the theory of symmetric spaces.
This notion was introduced for symmetric spaces in \cite{Chen-1988-RiemannianGeometricInvariantItsa}, 
and was generalized to quandles in \cite{Kubo-2022-CommutativityConditionSubsetsQuandlesa}.
% - def: antipodal set
\begin{dfn}\label{def:antipodal}
  Let $X$ be a quandle.
  \begin{enumerate}
    \item A subset $P \subset X$ is called a \emph{pole} if any two elements $x, y \in P$ satisfies that $s_x = s_y$.

    \item A subset $A \subset X$ is called an \emph{antipodal set} if any two elements $x, y \in A$ satisfies that $x \qop y = x$.
  \end{enumerate}
\end{dfn}

It is clear that a pole is an antipodal set.  
The notion of antipodal sets also relates to admissibility.

% - lem: admissibility vs antipodal set
\begin{prop}\label{prop:Ant_vs_Adm}
  Let $x$ and $y$ be elements in  a quandle $X$.
  If $\eta_X(x) = \eta_X(y)$ in $\As(X)$, then the subset $\{x, y\}$ is a pole.
  % \begin{enumerate}
  %   \item If $\eta_X(x) = \eta_X(y)$ in $\As(X)$ for $x, y \in X$, then the subset $\{x, y\}$ is a pole.
  %   \item A pole of $X$ is an antipodal set.
  % \end{enumerate}
\end{prop}
\begin{proof}
  Since the action of $\As(X)$ factors through that of $\Inn(X)$, the assertion holds.
  % The second one immediately follows by the first axiom of quandles.
\end{proof}
From \cref{prop:AlgConn_vs_Adm} and \cref{prop:Ant_vs_Adm}, we have the following corollary.
\begin{cor}\label{cor:geometric_condition}
  Let $X$ be a quandle.
  If any pole of $X$ and any connected component of $X$ intersect a single point, then the quandle $X$ is admissible.
\end{cor}

%% file: 3_galex.tex
\section{Admissibility for Generalized Alexander quandles}\label{sec:GAlex}

In this section, we investigate the generalized Alexander quandle and its admissibility.
The quandle is a group with a quandle structure given by a group automorphism (\cref{def:GAlex}). 

We now introduce some elementary properties for a generalized Alexander quandle.
In the following, we denote $X := \GAlex(G, \sigma)$, and the right multiplication of $g \in G$ as a group by $R_g$.
It is easy to show that $R_g \in \Aut(\GAlex(G, \sigma))$, and the map $R: G \to \Aut(\GAlex(G, \sigma))$ defined by $R(g) := R_g$ for $g \in G$ is an injective group homomorphism.
In particular, these quandles are homogeneous.
We denote the algebraic connected component of $X$ that contains the identity element $e \in G$ by $P$.
We note that the subset $P$ is a normal subgroup of $G$, and is a subquandle of $X$ (see \cite{Higashitani-2024-ClassificationGeneralizedAlexanderQuandlesa}).
% - prop:conn_GAlex
\begin{prop}\label{prop:conn_GAlex}
  The algebraic connected component containing $x \in X$ is equal to the left coset $xP$.
    % If two elements $x, y \in X$ are contained in the same algebraic connected component, then $x y^{-1} \in P$.
\end{prop}
\begin{proof}
  % First, we show $(1)$.
  Let us take $y \in \mathcal{O}_x$.
  Then, there exist $z_1, \dots, z_n \in X$ and $\varepsilon_i \in \{\pm 1\}$ such that 
  $y = ( \cdots (x \qop^{\varepsilon_1} z_1) \qop^{\varepsilon_2} \cdots \qop^{\varepsilon_{n-1}} z_{n-1}) \qop^{\varepsilon_n} z_n$.
  Here, we can assume $\sum_{i=1}^n \varepsilon_i = 0$ using the first axiom of quandles.
  Then, the following equation can be derived by induction on $n$:
  \begin{equation*}
    % \label{eq:GAlexP}
    ( \cdots (x \qop^{\varepsilon_1} z_1) \qop^{\varepsilon_2} \cdots \qop^{\varepsilon_{n-1}} z_{n-1}) \qop^{\varepsilon_n} z_n
    = x ( \cdots (e \qop^{\varepsilon_1} z_1) \qop^{\varepsilon_2} \cdots \qop^{\varepsilon_{n-1}} z_{n-1}) \qop^{\varepsilon_n} z_n \in x P.
  \end{equation*}
  Thus, we have $y \in xP$, and hence $\OO_x \subset xP$.
  Conversely, suppose that $y \in xP$.
 Using the same argument as before,
  it satisfies $y \in \OO_x$.
  Thus, we have $xP \subset \OO_x$, as desired.
  % Since the subset $P$ is a normal subgroup of $G$,
  % the assertion $(2)$ immediately follows by $(1)$.
  % This completes the proof.
\end{proof}

\begin{rmk}
  It is known that the normal subgroup $P$ is isomorphic to the \emph{displacement group} $\mathrm{Dis}(X)$, which is a group generated by $\{s_x \circ s_y^{-1} \mid x,y, \in X\}$.
\end{rmk}

Let $H$ be the subgroup consisting of all fixed points of the action of $\sigma$ on $G$, that is, 
\[
  H := \{g \in G \mid \sigma(g) = g\}.
\]
% - prop: H=max.anti.set
\begin{prop}\label{prop:max_anti_GAlex}
  % A subset $T \subset X$ containing $g \in X$ is antipodal if and only if $T \subset Hg$.
  The following hold:
  \begin{enumerate}
    \item[$(1)$] For a subset $T \subset X$ containing $g \in X$,
    the set $T$ is antipodal if and only if $T \subset Hg$.

    \item[$(2)$] A maximal antipodal set of $X$ containing $g \in G$ is equal to the subset $Hg$.

    \item[$(3)$] The subset $Hg$ is a pole of $X$.
    % If two elements $x, y \in X$ satisfy $\eta_X(x) = \eta_X(y)$, then $xy^{-1} \in H$.
  \end{enumerate}
\end{prop}
\begin{proof}
  We first show $(1)$.
  We assume that a subset $T$ containing $g$ is antipodal.
  Let us take $x \in X$.
  Then, it satisfies $x = x \qop g = \sigma(x g^{-1}) g$,
  and hence $xg^{-1} \in H$.
  Thus, we have $x \in Hg$.
  This shows that $T \subset Hg$.
  To show the inverse, we prove that the subset $Hg$ is antipodal.
  For $hg, h'g \in H g$, we have
  \[
    hg \qop^{\pm} h'g = \sigma^{\pm 1}(hg(h'g)^{-1}) h'g = \sigma^{\pm 1}(h h'^{-1})h'g = h h'^{-1}h'g = hg,
  \]
  as desired.

  Next, we show $(2)$.
  To show this, it suffices to show that $Hg$ is a maximal antipodal set by $(1)$.
  Suppose that $T$ is a trivial subquandle of $X$ with $Hg \subset T$.
  Let us take $t \in T$.
  Since $T$ is a trivial quandle containing $g \in T$, it satisfies the following:
  \[
    t = t \qop g = \sigma(tg^{-1})g.
  \]
  Thus, we have $\sigma(tg^{-1}) = tg^{-1}$.
  This shows $tg^{-1} \in H$, and hence $t \in Hg$.
  Therefore, we conclude that $T = Hg$, as desired.

  Finally, we show $(3)$.
  For $hg, h'g \in Hg$, we have
  \begin{align*}
      (hg) \qop (h'g) = \sigma((hg)(h'g)^{-1})h'g = \sigma(hh'^{-1})h'g = hh'^{-1}h'g = hg,
  \end{align*}
  which completes the proof.
\end{proof}

We now give a proof of the main theorem.
% - proof of main theorem
\begin{proof}[Proof of \cref{thm:main_thm}]
  Assume that $\eta_X(x) = \eta_X(y)$ for $x, y \in X$.
  Since $xP = yP$ and $Hx = Hy$ by \cref{prop:conn_GAlex}, \cref{prop:max_anti_GAlex} and \cref{cor:geometric_condition}, and the subset $P$ is a normal subgroup of $P$,
  we have $xy^{-1} \in P \cap H$.
  By the assumption, the intersection $P \cap H$ is equal to the trivial subgroup $\{e\}$ of $G$.
  Thus, we have $xy^{-1} = e$.
  Therefore, we obtain $x = y$, and this shows that the map $\eta_X: X \to \As(X)$ is injective.
  This completes the proof.
\end{proof}

\cref{thm:main_thm} can be regarded as a generalization of the result by Dhanwani--Raundal--Singh \cite{Dhanwani-2023-DehnQuandlesGroupsOrientable}.
They proved that if $H = \{e\}$, then the generalized Alexander quandle is faithful.
The condition $H = \{e\}$ means that the maximal antipodal set containing $e$ consists of a single point.
We now show that this condition also implies connectedness of the generalized Alexander quandle.
Consequently, the result of \cite{Dhanwani-2023-DehnQuandlesGroupsOrientable} can be viewed as a special case of our main result for connected generalized Alexander quandles.

\begin{prop}
  Let us define a map $\psi_\sigma: X \to P$ by $\psi_\sigma(g) := e \qop g = \sigma(g^{-1})g$.
  Then, the following hold:
  \begin{enumerate}
      \item[$(1)$] For elements $g_1, g_2 \in X$, we have $\psi_\sigma(g_1) = \psi_\sigma(g_2)$ if and only if $g_1 g_2^{-1} \in H$. 
      \item[$(2)$] If $H = \{e\}$ and the quandle $X$ is finite, then $X$ is connected.
  \end{enumerate}
\end{prop}

\begin{proof}
    To show $(1)$, suppose that $\psi_\sigma(g_1) = \psi_\sigma(g_2)$.
    Thus, we have $\sigma(g_1^{-1})g_1 = \sigma(g_2^{-1})g_2$, and hence $\sigma(g_1g_2^{-1}) = g_1g_2^{-1}$.
    This shows $g_1g_2^{-1} \in H$.
    Conversely we suppose $g_1g_2^{-1} \in H$. 
    Then $\{g_1, g_2\}$ is a pole subset in $X$.
    Therefore, we have
    \[
        \psi_\sigma(g_1) = e \qop g_1 = e \qop g_2 = \psi_\sigma(g_2),
    \]
    as desired.

    Next, we show $(2)$.
    Since $P \subset X$ is the connected component containing $e$, it is enough to show that $X \subset P$.
    By $(1)$ and the assumption, the map $\psi_\sigma: X \to P$ is injective.
    Thus, we have $|X| \leq |P|$.
    Since $X$ is a finite set, this shows that $X \subset P$.
    This completes a proof.
\end{proof}

The rest of this paper, 
we present examples of quandles whose admissibility cannot be determined by previous results, but is established by our main theorem. The computations below can be verified using the computer algebra system GAP \cite{GAP4} and the package Rig \cite{Rig}.

\begin{ex}\label{ex:GAlex_admissible}
Let $G$ be the group given by \texttt{SmallGroup(54,8)} in GAP. Then $|G| = 54$. Up to conjugacy, $G$ has exactly two automorphisms of order $8$, neither of which is an inner automorphism. Let $\sigma$ be one of them (the following argument is independent of this choice). We show that the generalized Alexander quandle $\GAlex(G,\sigma)$ is admissible by our main theorem. Since $G$ is non-abelian and $\sigma$ has exactly two fixed points, admissibility of $\GAlex(G,\sigma)$ cannot be determined by previous results.

We verify that $H \cap P = \{e\}$ for $\GAlex(G,\sigma)$ using GAP. First, we compute $\sigma$ as follows:
\begin{verbatim}
gap> G := SmallGroup(54,8);
gap> CC := ConjugacyClasses(AutomorphismGroup(G));
gap> order8_auts := Filtered(CC, f -> Order(Representative(f)) = 8);
gap> sigma := Representative(order8_auts[1]);
\end{verbatim}

Next, we construct the quandle $\GAlex(G,\sigma)$ using the package Rig:
\begin{verbatim}
gap> LoadPackage("rig");
gap> Q := HomogeneousRack(G, sigma);
\end{verbatim}

The object $Q$ corresponds to the generalized Alexander quandle $\GAlex(G,\sigma)$. Its elements are represented by the integers $1,2,\dots,54$. 
The matrix \texttt{Q.matrix} describes the multiplication of $Q$, where the $(i,j)$-entry is $j \qop i$.

We compute
\begin{verbatim}
gap> H := PositionsProperty(Q.matrix, i -> i = Q.matrix[1]);
gap> P := SortedList(Orbit(InnerGroup(Q), 1));
gap> Intersection(H, P);
\end{verbatim}

The output is \texttt{[1]}, which implies $H \cap P = \{e\}$. Therefore, $\GAlex(G,\sigma)$ is admissible by Theorem~\ref{thm:main_thm}.
\end{ex}

\cite{Higashitani-2024-ClassificationGeneralizedAlexanderQuandlesa} provides a complete classification of generalized Alexander quandles of order less than $128$. Among the $8726$ cases in this classification, we find seven examples of the type studied in this paper, including the case in Example \ref{ex:GAlex_admissible}. 
For all of these seven cases, admissibility cannot be determined by previous results, while it is confirmed by our main theorem. Their group orders are $54, 54, 81, 81, 96, 108, 108$ in ascending order.